\documentclass{amsart}                     
\usepackage{graphicx}
\usepackage{amsmath,amssymb,amsthm}

\newcommand{\CC}{\mathbb{C}}
 
\newcommand{\GG}{\mathbb{G}}
\newcommand{\HH}{\mathbb{H}}
\newcommand{\NN}{\mathbb{N}}
 
\newcommand{\RR}{\mathbb{R}} 
\newcommand{\XX}{\mathbb{X}}
\newcommand{\YY}{\mathbb{Y}}
\newcommand{\ZZ}{\mathbb{Z}}

\newcommand{\cH}{\mathcal{H}}
\newcommand{\cO}{\mathcal{O}}

\renewcommand{\a}{\alpha}
\newcommand{\D}{\Delta} 
\renewcommand{\d}{\delta} 

\newcommand{\g}{\gamma} 
\renewcommand{\L}{\Lambda} 
\renewcommand{\l}{\lambda}
\renewcommand{\b}{\beta} 
\renewcommand{\k}{\kappa} 
\newcommand{\Om}{\Omega}
\newcommand{\om}{\omega} 
 
\newcommand{\s}{\sigma}
  
\newcommand{\eps}{\varepsilon}
\newcommand{\dual}{\mathrm{d}}

\newcommand{\er}{\rangle}
\newcommand{\tr}{\mathrm{tr}}
\newcommand{\ghs}{\textsc{ghs}}
\newcommand{\fk}{\textsc{fk}}
\newcommand{\dlr}{\textsc{dlr}}

\theoremstyle{definition} \newtheorem{definition}{Definition}
\newtheorem{theorem}[definition]{Theorem}
\newtheorem{proposition}[definition]{Proposition}
\newtheorem{lemma}[definition]{Lemma}

\begin{document}

\title{Critical Value of the Quantum Ising Model on Star-Like Graphs}

\author{Jakob E. Bj{\"o}rnberg}
\address{Statistical Laboratory,
Centre for
Mathematical Sciences,
University of Cambridge,
Wilberforce Road, Cambridge CB3 0WB, U.K.}
\email{jeb76@cam.ac.uk}

\begin{abstract}
We present a rigorous determination of the critical value of the ground-state
quantum Ising model in a transverse field, on a class of planar graphs which
we call \emph{star-like}.  These include the star graph, which is a junction
of several copies of $\ZZ$ at a single point.
Our approach is to use the graphical, or {\fk-}, representation of the model,
and the probabilistic and geometric tools associated with it.
\end{abstract}

\maketitle

\section{Introduction}\label{intro}

The Hamiltonian of the quantum Ising model with transverse field on a finite
graph $G=(V,E)$ is the operator 
\begin{equation}
H=-\frac{1}{2}\l\sum_{e=xy\in E}\s_x^{(3)}\s_y^{(3)}-\d\sum_{y\in V}\s_x^{(1)}
\end{equation}
on the Hilbert space $\cH=\bigotimes_{x\in V}\CC^2$.  Here the Pauli
spin-$1/2$ matrices 
\begin{equation}
\s_x^{(3)}=
\begin{pmatrix} 
1 & 0 \\
0 & -1
\end{pmatrix},\qquad
\s_x^{(1)}=
\begin{pmatrix} 
0 & 1 \\
1 & 0
\end{pmatrix},
\end{equation}
and we use as basis for each copy of $\CC^2$ in $\cH$ the vectors 
$|+\er_x=\big(\begin{smallmatrix} 1 \\ 0\end{smallmatrix}\big)$ and 
$|-\er_x=\big(\begin{smallmatrix} 0 \\ 1\end{smallmatrix}\big)$;
also, $\l,\d>0$ are the spin-coupling and external-field intensities,
respectively.  Let $\b\geq0$ denote the inverse temperature, and define the
positive temperature states
\begin{equation}
\rho_{G,\b}(Q)=\frac{1}{Z_G(\b)}\tr(e^{-\b H}Q),
\end{equation}
where $Z_G(\b)=\tr(e^{-\b H})$ and $Q\in\CC^{2\times 2}$.  Also define the
\emph{ground state} to be the limit $\rho_G$ of $\rho_{G,\b}$ as
$\b\rightarrow\infty$.  If $G_n$ is an increasing sequence of graphs
tending to an infinite
graph $S$, then we may also speak of infinite-volume limits
$\rho_{S,\b}=\lim_{n\rightarrow\infty}\rho_{G_n,\b}$ and
$\rho_S=\lim_{n\rightarrow\infty}\rho_{G_n}$.  The existence of these limits is
discussed in~\cite{aizenman92:_percol_ising}.

In this article we will use the {\fk-} or random-cluster representation of the
ground state, see for example~\cite{ioffe_geom} and references therein.
Details will be provided in the 
next section, but roughly speaking the \fk-representation may be considered as
a limit of 
``discrete time'' random-cluster  models on $S\times(\eps\ZZ)$ as
$\eps\downarrow0$.   This is related to the well-known mapping of the quantum
Ising model onto the classical Ising model in one dimension
higher~\cite{sachdev99}, and the \fk-representation of that
model~\cite{grimmett_RCM}.  The relevance of this
representation is that it relates the occurrence of \emph{long range order} in
the ground state to the existence of infinite percolation paths in
$S\times\RR$;  here we say that the model exhibits long range order if for all
$x$, the correlation function 
\begin{equation}
G(x,y)=\rho_S(\s_x^{(3)}\s_y^{(3)})
\end{equation}
is bounded below by a positive function of $x$.  There is a
critical value of the ratio $\l/\d$ above which the model exhibits long range
order, and below which it does not.

The main result of this article is a rigorous determination of the critical
ratio for a certain class of planar graphs $S$ (see
Definition~\ref{main_def}).  This extends the calculation for the graph
$S=\ZZ$, to, amongst other graphs, the \emph{star graph}, which is the
junction 
of several copies of $\ZZ$ at a single point.  See Figure~\ref{star_fig}.  
A special case of our main result (Theorem~\ref{main_thm})  is therefore the
following. 
\begin{theorem}\label{star_thm}
The critical ratio for the ground state quantum Ising model on the star graph
is $\l/\d=2$.
\end{theorem}
\begin{figure}[hbt]
\centering
\includegraphics{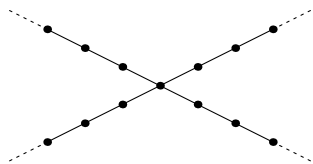}
\caption{The star graph has a central vertex of degree $k\geq3$ and $k$
  infinite arms, on which each vertex has degree 2.  In this illustration,
  $k=4$.}\label{star_fig} 
\end{figure}
In other words the critical ratio is the same for the star as for $\ZZ$;  this
is to be expected since the star is only \emph{locally} different from $\ZZ$.
We emphasise, however, that the class of graphs for which we prove this
result contains many more graphs than just the star.

The quantum Ising model on $\ZZ$ has been thoroughly studied, and the critical
ratio $\l/\d=2$ has been computed for this model in for
example~\cite{pfeuty70}.  See also~\cite{sachdev99} and references therein.
These calculations have relied on matrix methods and techniques such as
Jordan--Wigner transformation. 
Recently, in~\cite{bjogr2}, sharpness of the  
phase transition, and hence exponential decay of correlations below the
critical point, was established
rigorously for $G=\ZZ^d$ with any $d\geq 1$, using graphical methods similar
to the corresponding proof~\cite{abf} for the classical Ising model.  
Combining this result with 
duality arguments analogous to the classical two-dimensional random-cluster
model~\cite{grimmett_RCM}, this gives another proof that the critical ratio is
$\l/\d=2$ for this model, using only tools from stachastic geometry
(see~\cite{bjogr2} for details).   
One aim of this paper is to extend and illustrate the graphical
methods, and show how they can be applied to a wider range of structures than
just $\ZZ$.
The Ising model on the star-graph has also recently arisen in
the study of boundary effects in the two-dimensional classical Ising model,
see for example~\cite{trombettoni,martino_etal}.  Similar geometries have
also arisen in different problems in quantum theory, such as transport
properties of quantum wire systems,
see~\cite{chamon03,hou08,lal02}.

\section{Background and notation}\label{sec:1}

In this article we will let  $G=(V,E)$ be a \emph{star-like graph}:
\begin{definition}\label{main_def}
A star-like graph is a countably infinite connected planar graph, in which all
vertices have finite degree and only finitely many vertices have
degree larger than two.
\end{definition}
Such a graph is illustrated in Figure~\ref{graph_fig};  note that
the graph of Theorem~\ref{star_thm} is an example in which exactly one vertex
has degree at least three.
\begin{figure}[hbt]
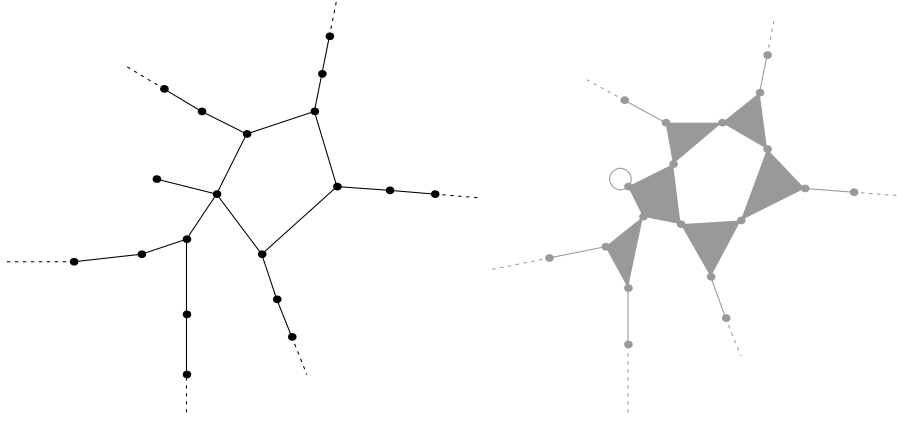

\centering
\includegraphics{starlike.2}
\includegraphics{starlike.3}
\caption{A star-like graph $G$ (left) and its line-hypergraph $H$ (right).
  Any vertex of degree $\geq3$ in $G$ is associated with a ``polygonal''
  hyperedge in $H$.}\label{graph_fig}
\end{figure}
Fix a planar embedding $\GG$ of $G$, and denote
$\XX=\GG\times\RR$;  also let $X=G\times\RR:=(V\times\RR,E\times\RR)$.  We
will sometimes use $X$ and $\XX$ interchangably.  Let $\cO$ be a fixed but
arbitrary vertex of $G$ of degree two or more, which we think of as the
origin.   

Recall that a \emph{hypergraph} is a set $W$ together with a collection $F$ of 
subsets of $W$, called \emph{edges};  a graph is a hypergraph in which all
edges contain two elements.  In our analysis we will use  a
suitably defined hypergraph ``dual'' of $\XX$:  let $H=(W,F)$ be the 
``\emph{line-hypergraph}'' of $G$, where $W=E$ and the set
$\{e_1,\dotsc,e_n\}\subseteq E=W$ is in $F$ 
if and only if $e_1,\dotsc,e_n$ are all the edges adjacent to some particular
vertex of $G$.  Note 
that only finitely many edges of $H$ have size larger than two.  There is a
natural planar embedding of $\HH$ defined via the embedding $\GG$, in which an
edge of size more than two is represented as a polygon.  See
Figure~\ref{graph_fig}.  Let $Y=H\times\RR$ and $\YY=\HH\times\RR$.

Our configuration space $\Om$ will be the set of pairs $\om=(B,D)$ where
$B\subseteq E\times\RR$ and 
$D\subseteq V\times\RR$ are \emph{locally finite}, which is to say that 
$B\cap (\{e\}\times[-n,n])$ and $D\cap (\{v\}\times[-n,n])$ are finite sets
for all $e\in E$, $v\in V$ and $n\in\NN$. 
We think of $B$ as a set of \emph{bridges} and $D$ as a set of \emph{deaths}
or cuts.  There is a natural embedding of any $\om\in\Om$ into $\XX$, where
deaths are represented as missing points and bridges as ``horizontal'' lines
connecting two ``vertical'' lines.  See Figure~\ref{config_fig} for an
illustration of this when $G=\ZZ$.  
\begin{figure}[hbt]
\centering
\includegraphics{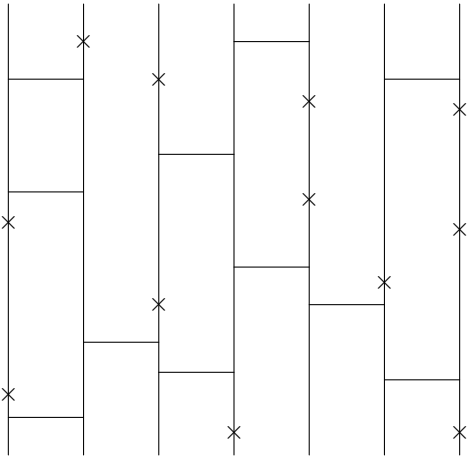}
\caption{A configuration $\om$ on $\ZZ\times\RR$.  Bridges are represented as
  horizontal line segments, and deaths as crosses.}\label{config_fig}
\end{figure}
Often we will identify $\om\in\Om$ with its embedding.  
Denote by $d(\cdot,\cdot)$ the graph distance in $G$, and let $\L_n$ denote
the set of points $(v,t)$ and $(e,t)$ where $v\in V$ has $d(v,\cO)\leq n$,
$e\in E$ has at least one endpoint at distance at most $n$ from $\cO$, and
$|t|\leq n$.  For each $\om\in\Om$, we will employ two \emph{restricted
  embeddings} $\om^1_n$ and $\om^0_n$, one ``wired'' and one ``free''.  The
free embedding $\om^0_n$ is simply the intersection of (the natural embedding
of) $\om$ with $\L_n$.  The wired embedding $\om_n^1$ is defined by
\begin{equation}
\om_n^1=\om_n^0\cup\{(v,t):d(v,\cO)=n+1, |t|\leq n\}
\cup\{(e,\pm n):e\in G_n\},
\end{equation}
where we have taken the liberty to identify $v\in V$ and $e\in E$ with their
embeddings in $\GG$.  In words, $\om_n^1$ is obtained by tying together the
top and bottom of $\om_n^0$, as well as all bridges portruding from its
``sides''.  We let the functions 
$k_n^0,k_n^1:\Om\rightarrow\NN$ count the number of connected components
of $\om_n^0$ and $\om_n^1$, respectively.

Equip $\Om$ with the Skorokhod topology and the associated
$\s$-algebra;  the details of their definitions are not immediately important,
but may be found in~\cite{bezuidenhout_grimmett} or~\cite{bjo_phd}.  Fix
$\l,\d>0$ and let $\mu=\mu_{\l,\d}$ 
be the probability measure on $\Om$ governed by a collection 
of independent Poisson processes $B_e$ on $\{e\}\times\RR$, for $e\in E$, and
$D_v$ on $\{v\}\times\RR$, for $v\in V$.  Here each $B_e$ has intensity $\l$,
each $D_v$ has intensity $\d$, and $B=\cup_{e\in E}B_e,D=\cup_{v\in V}D_v$.
This $\mu$ is the space-time (or ``continuum'') percolation measure
of~\cite{grimmett_stp}.  We may now define the random-cluster
probability measures. 
\begin{definition}\label{rcm_def}
The random-cluster measure $\Phi^b_n$ on $\L_n$ with parameters $\l,\d,q>0$
and boundary condition $b\in\{0,1\}$ is the probability measure on $\Om$ given
by 
\begin{equation}
\frac{d\Phi_n^b}{d\mu}(\om)\propto q^{k^b_n(\om)},\qquad\om\in\Om.
\end{equation}
\end{definition}

Let 
\begin{equation}
\theta^b=\Phi^b((\cO,0) \mbox{ lies in an unbounded component}).
\end{equation}
The following basic facts may be proved in a conventional manner, as
in~\cite[Theorem~5.5]{grimmett_RCM};  details for this particular model may be
found in~\cite{bjo_phd}.    
\begin{proposition}
Let $q\geq 1$.  The weak limits $\Phi^b:=\lim_{n\rightarrow\infty}\Phi^b_n$
exist, and enjoy a \emph{phase transition} in the sense  
that there is $\rho_c=\rho_c(q)\in(0,\infty)$, 
depending only on $q$ (and $G$), such
that $\theta^b=0$ if $\l/\d<\rho_c$ and $\theta^b>0$ if $\l/\d>\rho_c$.  We
call $\rho_c$ the \emph{critical value} of the random-cluster model on
$G\times\RR$.
\end{proposition}

The relevance of the space-time random-cluster measures to the
quantum Ising (or more generally quantum Potts) model is explained
in~\cite{aizenman92:_percol_ising}; 
in particular \emph{the ground state  quantum Potts model on
  $G$ exhibits long-range-order iff the corresponding random-cluster model has 
  $\theta^b>0$}.  
Hence, to investigate the phase-diagram of the quantum Ising model we will set
$q=2$ and focus on finding the critical value $\rho_c$  above which
percolation occurs. 

Let us say a few more words about the ``dual''
$\YY$ of $\XX$.   Given any configuration $\om\in\Om$, one may
associate with it a \emph{dual} configuration on $\YY$ by placing a death
wherever 
$\om$ has a bridge, and a (hyper)bridge wherever $\om$ has a death. 
This is illustrated in Figure~\ref{duality_fig}.
\begin{figure}[hbt]
\centering
\includegraphics{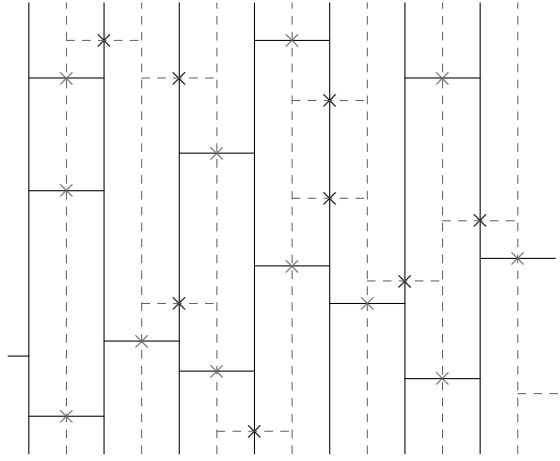}
\caption{Part of a configuration $\om$ (solid) and its dual $\om_\dual$
  (dashed with grey crosses) in the special case when
  $G=\ZZ$.}\label{duality_fig}  
\end{figure}
More precisely, we let $\Om_\dual$ be the set of pairs of locally finite
subsets of $F\times\RR$ and $W\times\RR$, and for each $\om=(B,D)\in\Om$ we
define its dual to be $\om_\dual:=(D,B)$.  As before, we may identify
$\om_\dual$ with its 
embedding in $\YY$, noting that some bridges may be embedded as polygons.  We
let $\Psi^b_n$ and $\Psi^b$ denote the laws of $\om_\dual$ under
$\Phi_n^{1-b}$ and $\Phi^{1-b}$ respectively.  

The case when  $G=\ZZ$ is particularly important,  and for this case we
use the lower case symbols $\phi$ and $\psi$ in place of $\Phi$ and $\Psi$,
respectively.  When $G=\ZZ$, the  dual space $\YY$ is isomorphic to $\XX$,
and we have the following result.  Again the proof is similar to that for the
discrete random-cluster model on $\ZZ^2$, but details for our model may be
found in~\cite{bjo_phd}. 
\begin{lemma}\label{planar_duality_lemma}
If $\phi_n^b,\phi^b$ have parameters $q$, $\l$ and $\d$, then the dual
measures $\psi_n^{1-b},\psi^{1-b}$ are random cluster measures with parameters
$q'=q$, $\l'=q\d$ and $\d'=\l/q$, and boundary condition $1-b$.
\end{lemma}

Recall that there is a partial order on
$\Om$ given by $(B',D')=\om'\geq\om=(B,D)$ if $B'\supseteq B$ and 
$D'\subseteq D$, and that an event $A$ is called \emph{increasing} if whenever
$\om\in A$ and $\om'\geq\om$ then also $\om'\in A$.  Also recall that $A$ is
called a \emph{cylinder event} if it only depends on a bounded region of
$\XX$, which is to say that there is a bounded set $\L\subseteq\XX$ such that
if $\om=\om'$ on $\L$ then $\om\in A$ if and only if $\om'\in A$.
\begin{definition}\label{basics_def}
Let $\k$ be a probability measure on $\Om$.  
\begin{itemize}
\item  We say that $\k$ is \emph{positively associated} if for  $A,B$ any
  increasing cylinder events,   $\k(A\cap B)\geq\k(A)\k(B)$.  
\item Another probability measure $\k_1$ on $\Om$ 
  \emph{stochastically dominates} $\k$ if for all increasing cylinder events
  $A$, we have $\k_1(A)\geq\k(A)$.  We write $\k_1\geq\k$.
\item We say that $\k$ has the \emph{positivity property} if for all
  $\eps>0$ there exists a constant 
  $0<c=c(\eps)<1$ such that for all $e\in E, v\in V, t\in\RR$,
  \begin{equation}
  c<\k(\mbox{no bridges in } \{e\}\times [t,t+\eps])<1-c
  \end{equation}
  and
  \begin{equation}
  c<\k(\mbox{no deaths in } \{v\}\times [t,t+\eps])< 1-c.
  \end{equation}
\end{itemize}
\end{definition}

\begin{proposition}
Let $q\geq1$.
The measures $\Phi^b_n,\Phi^b,\Psi^b_n,\Psi^b$ ($b=0,1$) are positively
associated and have the positivity property.  Moreover, $\Phi^1\geq\Phi^0$
and $\Psi^1\geq\Psi^0$.
\end{proposition}
The proof of this is similar to the discrete random-cluster model and is
omitted;  full details may be found in~\cite{bjo_phd}.

\section{The critical value}\label{sec:2}

\emph{We assume henceforth that $q=2$}.  It is known that,
if $G=\ZZ$, the critical value $\rho_c(2)=2$.  The following is the main
result of this paper. 
\begin{theorem}\label{main_thm}
Let $G$ be any star-like graph.  Then the critical value on $G\times\RR$ is
$\rho_c(2)=2$. 
\end{theorem}
In other words, the critical value for any star-like graph is the same as for
$\ZZ$.   Simpler arguments than those presented here can be used
to establish the analogous result when $q=1$, namely that $\rho_c(1)=1$.
Also, the same arguments can be used to calculate the critical value of the
discrete graphs $G\times\ZZ$ when $q=1,2$.  

Here is a brief outline of the proof of Theorem~\ref{main_thm}.  First we make
the straightforward observation that $\rho_c(2)\leq 2$.
Second, we use exponential decay and the {\ghs} inequality to establish the
existence of certain infinite paths in the dual 
model  when $\l/\d<2$.  Finally, we show how to put these paths together to
form  ``blocking circuits'' in $\YY$, which prevent the existence of infinite
paths in $\XX$ when $\l/\d<2$.  Parts of the argument are inspired
by~\cite{gkr}.

\begin{lemma}\label{upperbound_lem}  
For $G$ any star-like graph, $\rho_c(2)\leq 2$.
\end{lemma}
\begin{proof}
Any star-like graph $G$ contains an isomorphic copy of $\ZZ$ as a subgraph.
Let 
$Z$ be such a subgraph;  we may assume that $\cO\in Z$.  Also we let
$\phi^b_n,\phi^b$ denote the random-cluster measures on $Z\times\RR$.
For each $n\geq 1$, let $C_n$ be the event that in $\L_n$ there are no bridges
between $Z\times\RR$ and its 
complement.  Clearly each $C_n$ is a decreasing event.  It follows from a
standard property of random-cluster measures, sometimes called the
\dlr-property, that $\Phi^b_n(\cdot\mid C_n)=\phi^b_n(\cdot)$.  The proof of
this uses standard techniques~\cite{grimmett_RCM};  details for this model may
be found in~\cite{bjo_phd}.  If $A$ is an increasing cylinder
event, this means that
\begin{equation}
\phi^b_n(A)=\Phi^b_n(A\mid C_n)\leq\Phi^b_n(A),
\end{equation}
i.e. $\phi^b_n\leq\Phi^b_n$ for all $n$.  Letting $n\rightarrow\infty$ it
follows that $\phi^b\leq\Phi^b$.  If $\l/\d>2$ then 
$\phi^b((\cO,0)\leftrightarrow\infty)>0$ so then also 
\begin{equation}
\Phi^b((\cO,0)\leftrightarrow\infty)>0,
\end{equation}
which is to say that $\rho_c(2)\leq2$.
\end{proof}

\subsection{Infinite paths in the half-plane}

Let us now establish some facts about the random-cluster model on
$\ZZ_+\times\RR$ which will be useful later.  Our notation is as follows:  for
$n\geq 1$,
\begin{equation}
\begin{split}
S_n&=\{(a,t)\in\ZZ\times\RR: -n\leq a\leq n, |t|\leq n\}\\
S_n(m,s)&=S_n+(m,s)=\{(a+m,t+s)\in\ZZ_+\times\RR: (a,t)\in S_n\}.
\end{split}
\end{equation}
For brevity write $T_n=S_n(n,0)$;  also let $\partial$ denote the boundary,  
\begin{equation}
\partial S_n=\{(a,t)\in\ZZ\times\RR:a=\pm n\mbox{ or }t=\pm n\}
\end{equation}
and $\partial S_n(m,s)=\partial S_n+(m,s)$.
For $b=0,1$ and $\D$ one of $S_n,T_n$, we let $\phi^b_\D$ denote the
$q=2$ random-cluster measure on $\D$ with boundary condition $b$ and
parameters $\l,\d$.  Note that
\begin{equation}
\phi^b=\lim_{n\rightarrow\infty}\phi^b_{S_n},\qquad
\psi^b=\lim_{n\rightarrow\infty}\psi^b_{S_n}.
\end{equation}
We will also be using the limits
\begin{equation}
\phi^w=\lim_{n\rightarrow\infty} \phi^1_{T_n},\qquad
\psi^f=\lim_{n\rightarrow\infty} \psi^0_{T_n}.
\end{equation}
These are measures on configurations $\om$ on $\ZZ_+\times\RR$;  but according
to our definition they cannot be random-cluster measures since the regions
$T_n$ do not tend to the whole of $\ZZ\times\RR$.  However, standard arguments
let us deduce all the properties of $\phi^w,\psi^f$ that we need.  In
particular $\psi^f$ and $\phi^w$ are mutually dual (with the obvious
interpretation of duality) and they enjoy the positive
association and positivity properties of Definition~\ref{basics_def}.

Let $W$ be the ``wedge''
\begin{equation}
W=\{(a,t)\in\ZZ_+\times\RR: 0\leq t\leq a/2+1\},
\end{equation}
and write $0$ for the origin $(0,0)$.
\begin{lemma}\label{wedge_lem}  Let $\l/\d<2$.  Then
\begin{equation}
\psi^f(0\leftrightarrow\infty\mbox{ in } W)>0.
\end{equation}
\end{lemma}

Here is some intuition behind the proof of Lemma~\ref{wedge_lem}.  
The claim is well-known with $\psi^0$ in place of $\psi^f$, by standard
arguments using duality and exponential decay.  However, $\psi^f$ is
stochastically smaller than $\psi^0$, so we cannot deduce the result
immediately.  Instead we pass to the dual $\phi^w$ and establish directly a
lack of blocking paths.  The problem is the presence of the infinite ``wired
side'';  we get the required fast decay of two-point functions by using the
following result of~\cite{higuchi93_ii}, adapted to our model.

\begin{proposition}\label{exp_prop}
Let $\l/\d<2$.  There is $\a>0$ such that for all $n$,
\begin{equation}
\phi^1_{S_n}(0\leftrightarrow \partial S_n)\leq e^{-\a n}.
\end{equation}
\end{proposition}

In words, the two-point function decays exponentially also in finite 
volume.
Higuchi~\cite{higuchi93_ii}  proves a more general result for the
discrete Ising model, but attributes to Aizenman the simpler result  for that
model.  It was pointed out to us by Grimmett (personal communication) that the
original proof may be shortened by using the Lieb inequality in place of the
{\ghs} inequality, and we present the full proof for our model using the Lieb
inequality here.  Apart from the Lieb inequality, the proof uses another fact
known for  the $q=2$ Ising case but not for the general case $q\geq1$,
namely exponential decay in the infinite volume subcritical Gibbs state. 
\begin{proof}
Let $\overline S_n\supseteq S_n$ denote the ``tall'' box 
\begin{equation}
\overline S_n=\{(a,t)\in\ZZ\times\RR: -n\leq a\leq n, |t|\leq n+1\}.
\end{equation}
We will use a variant of the random-cluster measure on $\overline S_n$ which
has non-constant intensities for bridges and deaths, and also a process of
\emph{ghost-bonds}.  To this end we create a  new site $g$, which we think
of as a ``point at infinity'', and let
$\d(\cdot),\g(\cdot):\ZZ\times\RR\rightarrow\RR$ and
$\l(\cdot):(\ZZ+1/2)\times\RR\rightarrow\RR$ be bounded, nonnegative and
measurable functions.  Given independent Poisson processes of bridges and
deaths of rates $\l(\cdot)$ and $\d(\cdot)$, respectively, and of links to $g$
of rate $\g(\cdot)$, we may define random-cluster measures as in
Definition~\ref{rcm_def}, where now any components connected to $g$ are to be
counted as the same.  

The particular intensities we use are these.  Fix $n$, and fix
$m\geq0$, which we think of as large.  Let
$\l(\cdot)$, $\d(\cdot)$ and $\g_m(\cdot)$ be given by 
\begin{equation}
\begin{split}
\d(a,t)&=\left\{
\begin{array}{ll}
\d, & \mbox{if } (a,t)\in S_n \\
0, & \mbox{otherwise},
\end{array}
\right.\\
\l(a+1/2,t)&=\left\{
\begin{array}{ll}
\l, & \mbox{if } (a,t)\in S_n \mbox{ and } (a+1,t)\in S_n\\
0, & \mbox{otherwise},
\end{array}
\right.\\
\g_m(a,t)&=\left\{
\begin{array}{ll}
\l, & \mbox{if exactly one of } 
      (a,t) \mbox{ and } (a+1,t) \mbox{ is in } S_n\\ 
m, & \mbox{if } (a,t)\in \overline S_n\setminus S_n\\
0, & \mbox{otherwise}.
\end{array}
\right.
\end{split}
\end{equation}
In words, the intensities are as usual ``inside'' $S_n$ and in particular
there is no external field in the interior;  on the left and right sides of
$S_n$, the external field simulates the wired boundary condition;  and on top
and bottom, the external field simulates an approximate wired boundary (as
$m\rightarrow\infty$). 
We introduce another parameter $r\in[0,1]$, and let $\tilde \phi^r_{m,n}$
denote the random-cluster measure on $\overline S_n$ with intensities
$\l(\cdot),\d(\cdot),r\g_m(\cdot)$. 
Note that $\tilde\phi^0_{m,n}$ and $\phi^0_{S_n}$ agree on events defined on
$S_n$, for any $m$.  

Let $X$ denote $\overline S_n\setminus S_n$ together with the left and right
sides of $S_n$.
By the Lieb inequality, proved for the space-time Ising formulation of the
present model in~\cite{bjogr2} (see also~\cite{bjo_phd}), we have that
\begin{equation}
\tilde\phi^1_{m,n}(0\leftrightarrow g)\leq
e^{8\d}\int_X dx\;\tilde\phi^0_{m,n}(0\leftrightarrow x)
\tilde\phi^1_{m,n}(x\leftrightarrow g)\leq
e^{8\d}\int_X dx\;\tilde\phi^0_{m,n}(0\leftrightarrow x),
\end{equation}
since $X$ separates $0$ from $g$.  Therefore, by stochastic domination by the 
infinite-volume measure,
\begin{equation}
\tilde\phi^1_{m,n}(0\leftrightarrow g)\leq
e^{8\d}\int_X dx\;\phi^0(0\leftrightarrow x).
\end{equation}
All the points $x\in X$ are at distance at least $n$ from the origin.  By
exponential decay in the infinite volume, as proved in~\cite{bjogr2} using
similar methods to the discrete case~\cite{abf}, there is an absolute constant
$\tilde\a>0$ such that  
\begin{equation}
\tilde\phi^1_{m,n}(0\leftrightarrow g)\leq e^{8\d}|X|e^{-\tilde\a n}=
e^{8\d}(8n+2)e^{-\tilde\a n}.
\end{equation}
Now let $C$ be the event that all of $\overline S_n\setminus S_n$ belongs to
the connected component of $g$, which is to say that all points on   
$\overline S_n\setminus S_n$ are linked to $g$.  Then by the \dlr-property
of random-cluster measures the conditional measure
$\tilde\phi^1_{m,n}(\cdot\mid C)$ agrees with $\phi^1_{S_n}(\cdot)$ on events
defined on $S_n$.  Therefore
\begin{equation}
\begin{split}
\phi^1_{S_n}(0\leftrightarrow \partial S_n)&=
\tilde\phi^1_{m,n}(0\leftrightarrow \partial S_n\mid C)
=\tilde\phi^1_{m,n}(0\leftrightarrow g\mid C)\\
&\leq\frac{\tilde\phi^1_{m,n}(0\leftrightarrow g)}{\tilde\phi^1_{m,n}(C)}
\leq\frac{e^{8\d}}{\tilde\phi^1_{m,n}(C)}\cdot (8n+2)e^{-\tilde\a n}.
\end{split}
\end{equation}
Since $\tilde\phi^1_{m,n}(C)\rightarrow 1$ as $m\rightarrow\infty$ we conclude
that 
\begin{equation}
\phi^1_{S_n}(0\leftrightarrow \partial S_n)
\leq e^{8\d}(8n+2)e^{-\tilde\a n}.
\end{equation}
Since each $\phi^1_{S_n}(0\leftrightarrow \partial S_n)<1$ it is a simple
matter to tidy this up to get the result claimed.
\end{proof}

\begin{proof}[of Lemma~\ref{wedge_lem}]
Let $T=\{(a,a/2+1): a\in\ZZ_+\}$ be the ``top'' of the wedge $W$.  
We claim that
\begin{equation}
\sum_{n\geq 1}\phi^w((n,0)\leftrightarrow T \mbox{ in } W)<\infty.
\end{equation}
Once this is proved, it follows from the Borel--Cantelli lemma that with
probability one under $\phi^w$, at most
finitely many of the points $(n,0)$ are connected to $T$ inside $W$.  Hence
under the dual measure $\psi^f$ there is an infinite path inside $W$ with
probability one, and by the positivity- and positive association properties it
follows that 
\begin{equation}
\psi^f(0\leftrightarrow\infty\mbox{ in } W)>0,
\end{equation}
as required.

To prove the claim we note that, if $n$ is larger than some constant, then
the event ``$(n,0)\leftrightarrow T \mbox{ in } W$'' implies the event
``$(n,0)\leftrightarrow \partial S_{n/3}(n,0)$''.  The latter event, being
increasing, is more likely under the measure $\phi^1_{S_{n/3}(n,0)}$ than
under $\phi^w$.  But by Proposition~\ref{exp_prop},
\begin{equation}
\phi^1_{S_{n/3}(n,0)}((n,0)\leftrightarrow \partial S_{n/3}(n,0))
=\phi^1_{S_{n/3}}(0\leftrightarrow \partial S_{n/3})\leq e^{-\a n/3},
\end{equation}
which is clearly summable.
\end{proof}

\subsection{Proof of the main result}

We prove one more lemma about the half-plane before going on to the main
result.  
\begin{lemma}\label{circuits_lem}
Let $\l/\d<2$.  There exists $\eps>0$ such that for each $n$,
\begin{equation}
\psi^f((0,2n+1)\leftrightarrow (0,-2n-1)\mbox{ off } T_n)\geq \eps.
\end{equation}
\end{lemma}
\begin{proof}
Let $L_n=\{(a,n):a\geq0)\}$ be the horizontal line at height $n$, and let
$\eps>0$ be such that 
$\psi^f(0\leftrightarrow\infty\mbox{ in } W)\geq\sqrt\eps$.
We claim that 
\begin{equation}
\psi^f((0,-2n-1)\leftrightarrow L_{2n+1}\mbox{ off } T_n)\geq \sqrt\eps.
\end{equation}
Clearly $\psi^f$ is invariant under reflection in the $x$-axis, and standard
arguments~\cite[Theorem~4.19]{grimmett_RCM} imply that it is also invariant
under vertical translation.  Thus once the claim is proved we get that
\begin{multline}
\psi^f((0,2n+1)\leftrightarrow (0,-2n-1)\mbox{ off } T_n)\geq\\
\geq\psi^f((0,-2n-1)\leftrightarrow L_{2n+1}\mbox{ off } T_n\\
\qquad\mbox{ and } (0,2n+1)\leftrightarrow L_{-2n-1}\mbox{ off } T_n)
\geq (\sqrt\eps)^2,
\end{multline}
as required.  See Figure~\ref{paths_fig}.
\begin{figure}[hbt]
\centering
\includegraphics{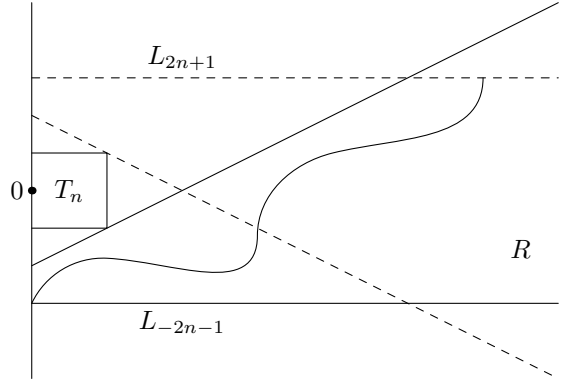}
\caption{Construction of a ``half-circuit'' in $\ZZ_+\times\RR$.  With
  probability one, any infinite path in the lower wedge must reach
  the line $L_{2n+1}$, and similarly for any infinite path in the upside-down
  wedge.  Any pair of such paths starting on the horizontal axis must
  cross.}\label{paths_fig}  
\end{figure}

The claim follows if we prove that
\begin{equation}\label{strip_eqn}
\psi^f(0\leftrightarrow\infty\mbox{ in } R)=0,
\end{equation}
where $R$ is the strip
\begin{equation}
R=\{(a,t):a\geq 0,-2n-1\leq t\leq 2n+1\}.
\end{equation}
However,~\eqref{strip_eqn} follows from the positivity property of
Definition~\ref{basics_def} and the Borel--Cantelli lemma, since the event
``no bridges between $\{k\}\times[-2n-1,2n+1]$ and 
$\{k+1\}\times[-2n-1,2n+1]$'' must happen for infinitely many $k$ with
$\psi^f$-probability one.  To see this we can compare $\psi^f$ with an
independent percolation measure, as in the proof of
Proposition~\ref{exp_prop}.  We have that $\psi^f\leq\mu$, where $\mu$ has
parameters $\l,\d$;  under $\mu$ the events above are independent, so
\begin{equation}
\psi^f(0\leftrightarrow\infty\mbox{ in } R)\leq
\mu(0\leftrightarrow\infty\mbox{ in } R)=0.
\end{equation}
\end{proof}
\begin{proof}[of Theorem~\ref{main_thm}]
We may assume that $G\neq\ZZ$, since the case $G=\ZZ$ is known.
Let $\l/\d<2$, and recall that $G$ consists of finitely many infinite
``arms'', where each vertex has degree two, together with a ``central''
collection of other vertices.  On each of the arms, let us fix one
arbitrary vertex (of degree two) and call it an \emph{exit point}.  Let $U$
denote the set of exit points of $G$.  

Given an exit point $u\in U$, call its
two neighbours $v$ and $w$;  we may assume that they are labelled so that only
$v$ can reach the origin $\cO$ without passing $u$.  If the edge $uv$ were
removed from $G$, the resulting graph would consist of two components, where
we denote by $J_u$ the component containing $w$.
Let $\hat \Phi^b_n,\hat \Phi^b$ denote the marginals of
$\Phi^b_n,\Phi^b$ on $X_u:=J_u\times\RR$;  similarly let
$\hat\Psi^b_n,\hat\Psi^b$ 
denote the marginals of the dual measures.  Of course $X_u$ is isomorphic to
the half-plane graph considered in the previous subsection.  
By positive association and the \dlr-property 
of random-cluster measures, $\hat\Phi_n^0\leq\phi^1_{T_n(u)}$, so letting
$n\rightarrow\infty$ also $\hat\Phi^0\leq\phi^w$.  Passing to the dual, it
follows that $\hat\Psi^1\geq\psi^f$.  The (primal) edge $uv$ is a
\emph{vertex} in the line-hypergraph;  denoting it still by $uv$ we therefore
have by Lemma~\ref{circuits_lem} that
there is an $\eps>0$ such that for all  $n$,
\begin{equation}
\Psi^1((uv,-2n-1)\leftrightarrow(uv,2n+1)
\mbox{ off }T_n(u)\mbox{ in }X_u)\geq\eps.
\end{equation}
Here $T_n(u)$ denotes the copy of the box $T_n$ contained in $X_u$.
Letting $A$ denote the intersection of the events above over all exit points
$u$, and letting $A_1=A_1(n)$ be the dual event
$A_1=\{\om_\dual:\om\in A\}$,  it follows from positive association that
$\Phi^0(A_1)\geq\eps^k$, where $k=|U|$  is the number of exit points.  Note
that $A_1$ is a decreasing event in the primal model.  The intuition is that
on $A_1$, no point in $T_n(u)$ can reach $\infty$ without passing the line 
$\{u\}\times[-2n-1,2n+1]$, since there is a dual blocking path in $X_u$.

Next let $I$ denote the (finite) subgraph of $G$ spanned by the complement of
all the $J_u$ for $u\in U$, and let $A_2=A_2(n)$ denote the event that for all
vertices  $v\in I$, the intervals $\{v\}\times[2n+1,2n+2]$ and 
$\{v\}\times[-2n-1,-2n-2]$
all contain at least one death and the endpoints of no bridges (in the primal
model).  By the positivity property, there is $\eta>0$ independent of $n$ such
that $\Phi^0(A_2)\geq\eta$. 
So by positive association $\Phi^0(A_1\cap A_2)\geq\eta\eps^k>0$.  On the
event $A_1\cap A_2$, no point inside the union of 
$I\times[-n,n]$ with $\cup_{u\in U}T_n(u)$ can lie on an infinite path.  See
Figure~\ref{blocking_fig}. 
\begin{figure}[htb]
\centering
\includegraphics{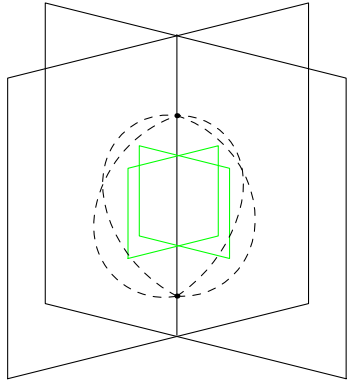}
\caption{The dashed lines indicate dual paths that block any primal connection
  from the interior to $\infty$.  Note that this figure illustrates only the
  simplest case when $G$ is a junction of lines at a single
  point.}\label{blocking_fig} 
\end{figure}
Taking the intersection of the $A_1(n)\cap A_2(n)$ over all $n$, it follows
that  
\begin{equation}
\Phi^0(\mbox{there is no unbounded connected component})\geq\eta\eps^k.
\end{equation}
The event that there is no unbounded connected component is a tail event.
All infinite-volume random-cluster measures are tail-trivial
(see~\cite[Theorem~4.19]{grimmett_RCM} or~\cite{bjo_phd}), so it follows,
whenever $\l/\d<2$, that  
\begin{equation}
\Phi^0(0\not\leftrightarrow\infty)=1.
\end{equation}
In other words, $\rho_c(2)\geq 2$.  Combined with the opposite bound in
Lemma~\ref{upperbound_lem}, this gives the result.
\end{proof}

\section*{Acknowledgements}
This research was carried out while the author was a Ph.D. student at
the University of Cambridge, UK, and the Royal Institute of Technology
(KTH), Sweden.  The author gratefully acknowledges funding from KTH during
this period, as well as generous support from Riddarhuset, Stockholm.
The author would also like to thank Geoffrey Grimmett for
many helpful discussions, to Petra Scudo for providing
references on the quantum Ising model, and to James
Norris for commenting on an early version of this article.

\bibliographystyle{plain}
\bibliography{../all}

\end{document}